\batchmode
\makeatletter
\def\input@path{{/Users/andrewpowell/Downloads/}}
\makeatother
\documentclass[oneside,english]{amsart}
\usepackage[T1]{fontenc}
\usepackage[latin9]{inputenc}
\usepackage{calc}
\usepackage{slashed}
\usepackage{amsthm}
\usepackage{amssymb}
\usepackage{wasysym}

\makeatletter

\providecommand{\tabularnewline}{\\}

\numberwithin{equation}{section}
\numberwithin{figure}{section}
\theoremstyle{plain}
\newtheorem{thm}{\protect\theoremname}
\theoremstyle{definition}
\newtheorem{defn}[thm]{\protect\definitionname}
\theoremstyle{remark}
\newtheorem{rem}[thm]{\protect\remarkname}
\theoremstyle{plain}
\newtheorem{lem}[thm]{\protect\lemmaname}
\theoremstyle{plain}
\newtheorem{cor}[thm]{\protect\corollaryname}

\makeatother

\usepackage{babel}
\providecommand{\corollaryname}{Corollary}
\providecommand{\definitionname}{Definition}
\providecommand{\lemmaname}{Lemma}
\providecommand{\remarkname}{Remark}
\providecommand{\theoremname}{Theorem}

\begin{document}
\title{A Universal Hypercomputer}
\author{Andrew Powell}
\address{Dr. Andrew Powell, Honorary Senior Research Fellow, Institute for
Security Science and Technology, Level 2 Admin Office Central Library,
Imperial College London, South Kensington Campus, London SW7 2AZ,
United Kingdom, }
\email{andrew.powell@imperial.ac.uk}
\begin{abstract}
This paper describes a type of infinitary computer (a hypercomputer)
capable of computing truth in the initial levels of the set theoretic
universe, \emph{V}. The proper class of such hypercomputers is called
a universal hypercomputer. There are two basic variants of hypercomputer:
a serial hypercomputer and a parallel hypercomputer. The set of computable
functions of the two variants is identical but the parallel hypercomputer
is in general faster than a serial hypercomputer (as measured by an
ordinal complexity measure). Insights into set theory using information
theory and a universal hypercomputer are possible, and it is argued
that the Generalised Continuum Hypothesis can be regarded as a information-theoretic
principle, which follows from an information minimization principle. 
\end{abstract}

\maketitle

\section{Introduction}

This paper introduces the notion of a universal hypercomputer and
shows that all sets in the Von Neumann hierarchy of pure sets can
be computed by a universal hypercomputer, and computation theory with
sufficient resources can be regarded as a recasting of set theory.
The significance of this equivalence is that there are likely to be
natural computational analogues in set theory. An example is given
of the Generalized Continuum Hypothesis, which is shown to be an information-theoretic
principle. and which follows from an information minimization principle
(see section \ref{sec:An-Information-Minimization}).\\
\\
According to B. J. Copeland \cite{Copeland2002} ``{[}a{]} hypercomputer
is any information-processing machine, notional or real, that is able
to achieve more than the traditional human clerk working by rote.''
Hypercomputers are a controversial topic (see \cite{Davis2003}, \cite{Davis2006})
because by definition they exceed what a human (or a computer) could
compute by rote with finite resources in a finite time. And certainly
it is not at all clear that you could physically build any kind of
hypercomputer (see \cite{Teuscher2002}). For example, an important
class of hypercomputer allows a computer to run forever and converge
to an output, and to start again with outputs taken to be inputs.
This type of hypercomputer (``a infinite run time hypercomputer'')
requires a countably infinite sequence of computation steps, which
humans cannot complete. Likewise a human could not load the input
registers of a hypercomputer which allows arbitrary real numbers as
input, because a human cannot load the uncountably infinitely many
bits in an arbitrary real number and, even if a real number could
be replaced by a finite label, there are uncountably many real numbers
(see \cite{BlumShubSmale89} and for a more recent survey see \cite{Ziegler2007}).
\\
\\
Having noted the impracticability of hypercomputers, by way of contrast
it is worth highlighting the long standing and rich literature of
(meta-)mathematical results describing the computational power of
different types of hypercomputer, starting with A. M. Turing's oracle
machines (see \cite{Copeland2002,Cleland2004,Ord2006}). Fundamentally
infinitary models of hypercomputers can provide strong intuitions
and sometimes result in simplifications of proofs and shortening of
the length of those proofs. An example, taken from the subject of
proof theory (see \cite{Rathjen2006} for this and other examples),
is that K Schütte's proof of the consistency of first-order Peano
arithmetic is much shorter than G. Gentzen's original (broadly finitary)\footnote{The principle of transfinite induction up to the countable ordinal
$\epsilon_{0}$ is less finitary than primitive recursive functions
but still corresponds to a definite progression in the complexity
of the (concrete) proof figures.} proof because Schütte introduced two natural inference rules with
an infinite number of premisses (usually countably infinitely many),
collectively known as the $\omega$-rule, which is generally the inference
from $S\vdash P(c)$ for all constant symbols $c$ to $S\vdash(\forall x)P(x)$,
and the dual inference from $S\vdash(\exists x)P(x)$ to $S\vdash P(c)$
for some constant symbol $c$). It is also worth mentioning that computational
power is related to the proof power of a deductive axiom system because
a (total) function $f$ is computable if $(\forall x)(\exists y)(f(x)=y)$
is provable in some deductive axiom system. It would seem to follow
that a hypercomputer can prove more than a Turing machine could; and
indeed, this is true. A hypercomputer which allowed countably infinitely
many registers of a computer to be non-empty and allowed a state to
require countably infinitely many register values to match a condition
would be able to implement the $\omega$-rule and to decide the truth
or falsehood of every proposition in first-order Peano arithmetic.
The difficulty is that the $\omega$-rule and the computer states
that correspond to it are in principle not human computable when any
model of a deductive system has an infinite domain (such as the set
of the natural numbers). But of course that does nothing to undermine
the truth of the  result that first-order Peano arithmetic plus the
$\omega$-rule is complete for the language of first-order arithmetic.
\\
\\
In a similar vein a number of important results are known about hypercomputers.
We will cite two such results. The first result (from \cite{HamkLew2000})
is that an infinite run time hypercomputer is complete for first-order
arithmetical truth and can decide the truth of all $\Pi_{1}^{1}$
propositions (\emph{i.e.} propositions of the form $(\forall X)P(X)$,
where formula $P$ may contain bounded variables over the natural
numbers but the variable $X$ over sets of natural numbers remains
free) and decide the membership of sets of natural numbers that are
defined by a $\Pi_{1}^{1}$ formula with a free natural number variable.\footnote{It is possible for a $\Pi_{1}^{1}$ formula to have free variables
over sets of natural numbers (indeed that is an essential part of
the language of second order arithmetic), but in terms of a predicative
concept of set, one starts with sets of natural numbers definable
by arithmetical formulas and then defines sets of natural numbers
inductively by relativizing quantifiers over sets of natural numbers
of arbitrary formulas of second order arithmetic to the sets already
defined and iterates this construction to the first non-recursive
ordinal (see \cite{Feferman1964}). The resulting set of sets of natural
numbers are the hyperarithmetical sets of natural numbers.} The significance of this result is that $\Pi_{1}^{1}$ propositions
and sets are impredicative\footnote{A set of natural numbers defined by an impredicative formula is a
set of natural numbers defined by a formula that quantifies over all
sets of natural numbers.}, and by a classic result due to S.C. Kleene and C. Spector (see \cite{Feferman1964},
\cite{Sacks1990}) the sets of natural numbers (or equivalently real
numbers) defined by a $\Delta_{1}^{1}$ formula with a free natural
number variable\footnote{A $\Delta_{1}^{1}$ formula can be expressed in the form $\Pi_{1}^{1}$
and $(\exists X)Q(X)$, where formula $Q$ may contain bounded variables
over the natural numbers but the variable $X$ over sets of natural
numbers remains free.} can be identified with sets of natural numbers computable by a transfinite
sequence of oracle machines up to the first non-recursive ordinal,
starting from a universal Turing machine and adding a function which
computes the halting problem of the previous oracle machines in the
sequence (see \cite{HamkLew2000} and compare the infinite time register
machine defined in \cite{Koepke2006a}).\footnote{In fact an infinite run time hypercomputer can decide propositions
which extend up the analytical hierarchy and can be defined by a $\Delta_{2}^{1}$
formula, see \cite{HamkLew2000} Theorem 2.5.} The fact that $\Delta_{1}^{1}\subset\Pi_{1}^{1}$ shows just how
powerful a hypercomputer must be to decide the truth of all $\Pi_{1}^{1}$
propositions or membership of $\Pi_{1}^{1}$ sets. It is also worth
mentioning that hyperarithmetical sets have been generalised by R.
Shore, G. Sacks \emph{et al} to set theory by means of $\alpha$-recursion
theory\footnote{$\alpha$ is an ordinal such that cumulative $L_{\alpha}$ of Gödel's
constructive universe of sets is a model of Kripke-Platek set theory.}. P. Koepke and B. Seyfferth \cite{KoepSeyf2008} have shown that
hypercomputers with $\alpha$ registers and up to $\alpha$ steps
in a computation with a finite program can compute $\alpha$-recursive
and $\alpha$-recursively enumerable sets, and can be used to prove
results in $\alpha$-recursion theory computationally. A second, even
stronger result (from P. Koepke, see \cite{Koepke2005,Koepke2006,Koepke2009})
is that a hypercomputer that has a finite program, but has an infinite
number of registers and an infinite run time that can have any infinite
ordinal value, can compute all constructible sets (in the sense of
K. Gõdel's constructible universe of sets, see \cite{Kunen1980} for
a clear introduction) of ordinals from finitely many ordinal parameters.
This result shows that ordinal constructibility (or better definability
in terms of previously defined sets) is the same as a general notion
of ordinal computability with a finite program. \\
\\
Now although the literature has considered Turing machines/register
machines with infinite run time (which always terminate after countably
many steps) and Turing machines/register machines with infinite run
time and infinite memory indexed by the class of all ordinals (known
as o\emph{rdinal computers}), there has been no exploration to date
of Turing machines with infinite run time, infinite memory and programs
with an infinite number of instructions. This paper proves the result
that the set of hypercomputable sets with finitely many ordinal parameters
(specifying the hypercomputer configuration) is the Von Neumann hierarchy
of pure sets. \\
\\
In many ways this result is fairly obvious: unconstrained computation
resources lead to every set being computable. But it also leads to
the thought that computational notions are likely to have natural
set theory analogues. If we define the number of bits of information
in a set $x$ (expressed as a binary sequence that represents all
the members of $x$ as well as $x$) as the least length of the sequence
which can be losslessly compressed from $x$, then we can see that
the number of bits of information in a binary sequence of length $\alpha$
is $\le\alpha$. In fact the amount of information in a set is a cardinal
number, $\aleph$, because any sequence of length $\aleph\le\alpha<\aleph+1$
can be losslessly compressed by being mapped one-to-one and onto a
sequence of length $\aleph$ by definition of cardinal number. It
is shown in Theorem \ref{thm:GCH-is-equivalent} that the Generalized
Continuum Hypothesis (GCH) states that the amount of information needed
to decide the relation $x\in X$ by enumeration\footnote{The enumeration is an interleaved enumeration of $X$ and $2^{\aleph}-X$.}
of $X\subseteq2^{\aleph}$ and $2^{\aleph}-X$ is $<\aleph+1$, where
$x$ is expressed as a binary sequence of length at most cardinal
$\aleph\ge\aleph_{0}$ and $2^{\aleph}$ is the set of all such binary
sequences. Of the standard principles of Zermelo Fraenkel set theory,
GCH is the only principle that can be cast in an explicitly information-theoretic
way, but the Axiom of Separation and the Axiom Schema of Replacement
limit the information in a set by limiting its size, and the Axiom
of Foundation ensures that a set has a bounded amount of information
(because every membership chain must terminate after finitely many
steps). \\
\\
There is a view that second-order Zermelo Fraenkel set theory and
the universal hypercomputer that computes its unique class model\footnote{There is of course a hierarchy of set models of second-order Zermelo
Fraenkel set theory defined by taking the set theoretic universe,
$V$, up to the level of each uncountable strongly inaccessible cardinal,
see \cite{Hellman89} for an interesting discussion of a modal-structural
view of set theory. }, $V$, are too powerful to be useful in mathematics. It is shown
in Corollary \ref{cor:GCH} below that a universal hypercomputer computes
GCH as true in $V$ if an information minimization principle is true,
by exploiting the link between $V$ and the universal hypercomputer,
\emph{i.e.} that $V$ is ``the class of'' the universal hypercomputer
and the universal hypercomputer is ``the computer of'' $V$. The
information minimization principle states that, for losslessly incompressible
sets, to any hypercomputation that decides $x\in X$ by enumeration
of $X$ and its complement there corresponds a hypercomputation that
decides $x\in X$ by enumeration of $X$ and its complement that has
the length of the minimum number of bits of information in $x\in X$
and $x\notin X$. This information minimization principle is an expression
of the fact that all sets and all membership relations can be hypercomputed
and that a set and a relation contain a certain number of bits of
information, and it does not matter how  those bits are enumerated,
as some enumeration of this number of bits will define the set and
decide the truth of the relation for particular sets. One strong assumption
in this argument is that all hypercomputations can be performed in
the universe of associated sets (which can be mapped one-to-one and
onto $V$, see Theorem \ref{thm:minimum-info}). We can also say that
it is assumed that $2^{\aleph}$ exists and that a corresponding $\langle2^{\aleph},2^{\aleph},2^{\aleph}\rangle$-hypercomputer
exists. These assumptions are equivalent to the existence and uniqueness
of $V$. It is of course possible to identify a set $X\subseteq2^{\aleph}$
by means of a particular formula or predicate in $>\aleph$ bits if
a quantified variable in the formula ranges over say $\subseteq2^{2^{\aleph}}$,
but the set itself in $V$ does not change and is still $\subseteq2^{\aleph}$.
\\
\\
We could in fact define a set $X$ of cardinality $\le2^{\aleph}$
as a set of sets $x$ that can be defined in $\le\aleph$ bits by
enumeration such that the membership relation between $x$ and $X$
(see Theorem \ref{thm:minimum-info}) can also be decided in $\le\aleph$
bits by enumeration. The basic argument for GCH is that GCH is equivalent
to the statement that $x\in X\subseteq2^{\aleph}$ can be decided
by enumeration almost always in $<\aleph+1$ steps for infinite cardinal
$\aleph$ (see Theorem \ref{thm:GCH-is-equivalent}) and yet this
statement is equivalent to the claim that the number of bits of information
in the relation $x\in X$ is $\aleph$. 

\section{What is a Universal Hypercomputer?}

So far we have not defined a universal hypercomputer. We start with
a hypercomputer that can compute truth in the initial levels $\le\alpha$
of von Neumann cumulative hierarchy of pure, well founded-sets, \emph{V}.
The proper class of all such hypercomputers can compute the truth
of all propositions in \emph{V} and forms \emph{a universal hypercomputer}.
This type of hypercomputer permits programs of infinite ordinal length,
infinitely many registers and computations of infinite length, which
is possible if the registers are left in a consistent state at limit
ordinals during computations. In the following definitions we split
out the number of registers, the length of computations and and the
length of the program as separate parameters.
\begin{defn}
A \emph{$\langle\aleph,\beth,\daleth\rangle$-hypercomputer}, for
cardinals $\aleph$ and $\beth$ and ordinal $\daleth$, where $\beth\leq\daleth\leq\aleph$,
comprises the following elements:
\end{defn}

\begin{itemize}
\item \emph{$\aleph$-many Registers} for storage of inputs, outputs and
workings of a computation. For ease of exposition\footnote{Separate input, working and output registers are not essential, as
registers can always be moved around and working space created, but
I hope their use makes the exposition easier to follow.} there will be disjoint sets of registers for inputs, outputs and
workings. Input registers are read-only and contain inputs in the
hypercomputer's initial state. Working registers are read-write and
receive a copy of the inputs when the program starts. Output registers
receive a copy of the content of the working registers, are write-only
by the program and contain the outputs of the program in the hypercomputer's
halting state (see below). A register consists of an ordinal identifier
and a data field, written $R_{\alpha}$ for $\alpha<\aleph$, which
can contain 0 or 1. By default all registers are initialized with
the value 0 (representing ``empty''). Input registers will be written
$I_{\alpha}$, working registers $W_{\alpha}$, and output registers
$O_{\alpha}$. It is convenient to allow multiple disjoint sets of
working registers, $W_{\beta,\alpha}$, to facilitate operations on
data set,\footnote{Disjoint sets of registers can be reproduced by coding the set of
disjoint sequences $\{\langle a_{1,\alpha},a_{2,\alpha}\cdots,a_{i<\aleph,\alpha},\cdots\rangle:\:\alpha<\aleph\}$
by the concatenation $\langle a_{1,1},a_{2,1},\cdots,a_{i<\aleph,1},\cdots\rangle\langle0,1\rangle\langle a_{1,2},a_{2,2},\cdots,a_{i<\aleph,2},\cdots\rangle\langle0,1\rangle\cdots\langle0,1\rangle\langle a_{1,\alpha},a_{2,\alpha},\cdots,a_{i<\aleph,\alpha},\linebreak\cdots\rangle\cdots$
with the marker $\langle0,1\rangle$ placed after each successor and
limit member of the sequence and having rules to skip over markers. } and it will be assumed in this paper that working registers are partitioned
into disjoint sets.\footnote{\label{fn:9}For example, take a program which has two states (other
than the standard special states), 1 and 2, the standard introduction
and conclusion for input and output being ignored for simplicity.
In state 1 if the program reads a register $W_{1,\alpha}$ containing
a 1, it writes a 1 in register $W_{2,1}$ and stays in state 1. In
state 1 if the program reads a register $W_{1,\alpha}$ containing
a 1, it moves right to $W_{1,\alpha+1}$ and stays in state 1, while
if $W_{1,\alpha}$ contains a 0 it writes a 0 in register $W_{2,1}$
and terminates by moving to the halting state, 2. When reading registers
$W_{1,\lambda}$ with limit ordinal $\lambda$, the program will be
in the highest state achieved (\emph{i.e.} 1 in practice) when reading
registers $W_{1,\alpha<\lambda}$ and the value of any register $W_{\beta,\alpha\leq\lambda}$
after limit ordinal $\lambda$ steps of the program will be the value
of an eventually constant sequence $W_{\beta,\alpha}$ for $<\lambda$
steps or 1 otherwise. It can be seen that $W_{2,1}$ contains 1 if
and only if every $W_{1,\alpha}$ for ordinal $\alpha<\aleph$ contains
1. The program implements infinite logical conjunction (\emph{i.e.
}infinite logical ``and'') of propositions with truth values stored
in $W_{1,\alpha}$. This program can be written formally as follows
in the notation of this paper: $\langle1,W_{1},1,\langle2,2\rangle,1\rangle$,
$\langle1,W_{1},1,\langle8,1\rangle,1\rangle$, $\langle1,W_{1},0,\langle1,2\rangle,2\rangle,\langle1,W_{0},0,\langle12,0\rangle,1\rangle,\langle1,W_{0},1,\langle12,0\rangle,1\rangle,\langle1,W_{0},1,0,2\rangle$.
The last three instructions implement the flag set to 1 in $W_{0,0}$
when the program completes, and moves the program to the halt state.
Infinite ``or'' can be done similarly with the two state machine:
in state 1, if the program reads a register $W_{1,\alpha}$ containing
a 0 it writes a 0 to $W_{2,1}$, moves to $W_{1,\alpha+1}$ and stays
in state 1; if it reads a register $W_{1,\alpha}$ containing a 1
it writes a 1 to $W_{2,1}$ and moves to halting state 2.} To avoid complexities associated with the computability of functions
that jump between registers, registers perform like infinite linear
tapes of length $\aleph$ terminated on the left, with $R_{1}$ being
the register with lowest ordinal and only registers $R_{\alpha+1}$
and $R_{\alpha-1}$, where they exist, being accessible from $R_{\alpha}$.
$W_{0,0}$ is treated as a special register as it is set to\emph{
0} by default and set to \emph{1} if a program (or subprogram) runs
to completion, after $o(\aleph)$ steps, where $o(\aleph)$ is the
least ordinal of cardinality $\aleph$. This register can be used
as a ``flag'' to capture the output of the program.
\item \emph{Symbols} 0 and 1.
\item \emph{$\beth$}-many\emph{ States} which determine which action the
hypercomputer takes and any output it produces. A state can be identified
by an ordinal. There are at least two special states, an \emph{initial
state, }identified by the ordinal \emph{0,} where a program (see below)
starts and a \emph{halting state} where a program stops. The hypercomputer
enters the halting state, \emph{i.e.} stops, when none of the instructions
(see below) applies, or when the computation length is reached (when
the contents of $W_{0,0}$ are set to 1). Ordinary states are like
line numbers in a hypercomputer program (see \cite{Koepke2005}),
so from the initial state the program will enter the first ordinary
state, \emph{1} say, and as the number of instructions executed (\emph{i.e.}
the length of the computation) increases towards limit ordinal $\alpha$,
the program jumps to state $\alpha$ unless there is a state with
a smaller least upper bound.\footnote{It is of course possible to become stuck in a particular state and
for the program not to output given a particular set of register values,
but equally it is possible to loop back to the same state if the register
value is \emph{0} say, and then at the next limit ordinal for the
program to read a \emph{1}, when the program may move to a different
state.} It makes sense not to be able to jump past a limit ordinal, so for
successor ordinal state $\alpha$ only states with ordinal $prevlim(\alpha)\leq\beta<nextlim(\alpha)$
are accessible from $\alpha$, where $prevlim(\alpha)$ is the preceding
limit ordinal $\leq\alpha$ and $nextlim(\alpha)$ is the next limit
ordinal $>\alpha$. 
\item An\emph{ initial configuration}, comprising data loaded into the input
registers, an initial state and an initial current register ($I_{1}$
by default and likewise $W_{\beta,1}$ and $O_{1}$ when these sets
of registers are accessed).
\item A \emph{program} of length $\beth$\footnote{The instructions can be grouped by state into a table of instructions.
For ease of exposition, the program length will refer to the number
of state entries in the table.} which is a (in general transfinite) sequence of 5-tuples \emph{$\langle$Current
State, Register Set, Symbol, Action, Next State}$\rangle$, called
\emph{program instructions}, read as ``if the hypercomputer is in
Current State and the current register in the Register Set contains
Symbol then do Action and move into Next State'', where an Action
may be to do nothing, write a 0 or 1 to a current register, $R_{\alpha}$,
in any set of registers, to move left or right where possible, \emph{i.e.}
from $R_{\alpha}$ to $R_{\alpha-1}$ or $R_{\alpha+1}$ if $\alpha$
is a successor ordinal and from $R_{\alpha}$ to $R_{\alpha+1}$ otherwise,
or set the current register to the 0-th register,\emph{ i.e.} $R_{0}$.
As these operations apply to each disjoint set of registers, \emph{I},
\emph{$W_{\beta}$}, \emph{O}, there are 11 instruction types (as
``do nothing'' applies to all registers and \emph{I} cannot be written
to). For definiteness, ``do nothing'' can be represented by 0, ``write
a 0'' to the current register of $W_{\beta}$ by $\langle1,\beta\rangle$,
``write a 0'' to the current register of \emph{O} by $5$, ``write
a 1'' to the current register of $W_{\beta}$ by $\langle2,\beta\rangle$,
``write a 1'' to the current register of \emph{O} by $6$, ``move
left'' by 3 (for \emph{I}), $\langle7,\beta\rangle$ (for \emph{$W_{\beta}$})
and 9 (for \emph{O}), ``move right'' by 4 (for \emph{I}), $\langle8,\beta\rangle$
(for \emph{$W_{\beta}$}) and 10 (for \emph{O}), and ``reset register''
by 11 (for \emph{I}), $\langle12,\beta\rangle$ (for \emph{$W_{\beta}$})
and 13 (for \emph{O}). Each program comprises a standard introduction
which copies the input registers to working registers (\emph{i.e.}
a set of 5-tuples with source set of registers $I$ and destination
set of registers $W_{1}$),\footnote{A program to copy the registers from $I$ and destination $W_{1}$
has one ordinary state, 1, and comprises the instructions $\langle0,I,0,0,1\rangle$,
$\langle0,I,1,0,1\rangle$,$\langle1,I,0,\langle1,1\rangle,1\rangle$,
$\langle1,I,1,\langle2,1\rangle,1\rangle$, $\langle1,I,0,4,1\rangle$,
$\langle1,I,1,4,1\rangle$, $\langle1,I,0,\langle8,1\rangle,1\rangle$
$\langle1,I,1,\langle8,1\rangle,1\rangle$. The sequence $\langle1,W_{0},0,\langle12,0\rangle,1\rangle,\langle1,W_{0},1,\langle12,0\rangle,1\rangle,\langle1,W_{0},1,0,2\rangle$
will move the program to the halting state, 2, when it completes copying.} a program that manipulates the working registers, and a standard
conclusion which copies working registers to output registers (\emph{i.e.}
a set of 5-tuples with source registers $W_{\beta}$ and destination
registers $O$).\footnote{\label{fn:11}A program to copy the registers from $W_{\beta}$ and
destination registers \emph{O} has one ordinary state, 1, and comprises
the instructions $\langle0,W_{\beta},0,0,1\rangle$, $\langle0,W_{\beta},1,0,1\rangle$,
$\langle1,W_{\beta},0,5,1\rangle$, $\langle1,W_{\beta},1,6,1\rangle$,
$\langle1,W_{\beta},0,\langle8,\beta\rangle,1\rangle$, $\langle1,W_{\beta},1,\langle8,\beta\rangle,1\rangle$,
$\langle1,W_{\beta},0,9,1\rangle$, $\langle1,W_{\beta},1,9,1\rangle$.
The sequence $\langle1,W_{0},0,\langle12,0\rangle,1\rangle,\langle1,W_{0},1,\langle12,0\rangle,1\rangle,\langle1,W_{0},1,0,2\rangle$
will move the program to the halting state, 2, when it completes copying.} It is not possible for humans to write down infinitely long programs,
but it is possible to write \emph{program schemas}. An example is
a program schema for the logical conjunction of a set a registers
of cardinality $\aleph$ given by a finite program in footnote \ref{fn:9}
could be written $\langle1,W_{1},1,\langle2,2\rangle,1\rangle$, $\langle\alpha,W_{1},1,\langle8,1\rangle,\alpha+1\rangle$,
$\langle\alpha,W_{1},0,\langle1,2\rangle,o(\aleph)\rangle$, where
$\alpha<o(\aleph)$ is an ordinal parameter for the state and $o(\aleph)$
is the halt state.\footnote{Program schemas are concise, but finite programs suffice in the theorems
below except for writing data input and output, where most data will
need to be hard coded because there are only countably many program
schemas if each $\alpha$has no be defined by a finite formula.}
\item $\daleth$ many steps in the computation (see Definition \ref{def:2.2}). 
\item \emph{Output} is the contents of the output registers when the program
is in a halting state. 
\end{itemize}
A hypercomputer will read a program, which will start in the initial
state, run through its computation and terminate when it reaches a
halting state. The output of the program is the contents of the hypercomputer's
output registers.
\begin{defn}
\label{def:2.2}A \emph{computation} is a sequence of steps of length
$\daleth$ that results in output given specific input. 
\end{defn}

To make this characterisation precise, a computation can be considered
to take place in discrete time intervals indexed by ordinals. Following
\cite{Koepke2005} a ``step'' can be taken to have three components:
the current state at time $\alpha$, written $S_{\alpha}(R)$, a pointer
to the ordinal index of the current register, $H_{\alpha}(R)$, and
the contents of all the registers (a ``snapshot'' of the computation),
$C_{\alpha}(R):\daleth\rightarrow\{0,1\}$, where \emph{R} is a set
of registers \emph{I}, \emph{$W_{\beta}$}, \emph{O}. Limit ordinal
``steps'' are special, as the principle (see \cite{HamkLew2000})
will be adopted that if $S_{\alpha}(R)$, $H_{\alpha}(R)$ or $C_{\alpha}(R)$
are eventually constant for $\alpha<\lambda$, where $\lambda$ is
a limit ordinal, then by default $S_{\lambda}(R)$, $H_{\lambda}(R)$
or $C_{\lambda}(R)(\zeta)$ for $\zeta<\aleph$ will take those constant
values or else will take the limit of the least upper bounds, which
will be $C_{\lambda}(R)$($\zeta)=1$ if $C_{\alpha<\lambda}(R)$($\zeta)$
is not eventually constant and $S_{\lambda}(R)=\lambda$ and $H_{\lambda}(R)=R_{\lambda}$
if $S_{\lambda}(R)$ and $H_{\lambda}(R)$ are otherwise unbounded.
This is the ``lim sup'' construction (\emph{i.e}. the limit of the
least upper bounds). Recursive definitions for $S_{\alpha}(R)$, $H_{\alpha}(R)$
and $C_{\alpha}(R)$ are given as follows (again based on \cite{Koepke2005})\footnote{Koepke uses ``lim inf'' rather than ``lim sup'' because the programs
he considers are finite, and it makes no sense to jump to an infinite
limit ordinal state.}.\\

If $\langle\beta,R,b,a,\gamma\rangle$ is the instruction such that
$S_{\alpha}(R)=\beta$ and $C_{\alpha}(R)(H_{\alpha})=b$ then:
\begin{itemize}
\item $S_{0}(R)=0$
\item $S_{\alpha+1}(R)=\gamma$ where $prevlim(\alpha)\leq\gamma<nextlim(\alpha)$
\item $H_{0}(R)=0$
\item $H_{\alpha+1}(I)=H_{\alpha}(I)-1$ if $a=3$ and $H_{\alpha}(I)$
is a successor ordinal
\item $H_{\alpha+1}(W_{\beta})=H_{\alpha}(W_{\beta})-1$ if $a=\langle7,\beta\rangle$
and $H_{\alpha}(W_{\beta})$ is a successor ordinal
\item $H_{\alpha+1}(O)=H_{\alpha}(O)-1$ if $a=9$ and $H_{\alpha}(O)$
is a successor ordinal
\item $H_{\alpha+1}(I)=H_{\alpha}(I)+1$ if $a=4$
\item $H_{\alpha+1}(W_{\beta})=H_{\alpha}(W_{\beta})+1$ if $a=\langle8,\beta\rangle$ 
\item $H_{\alpha+1}(O)=H_{\alpha}(O)+1$ if $a=10$
\item $H_{\alpha+1}(I)=0$ if $a=11$
\item $H_{\alpha+1}(W_{\beta})=0$ if $a=\langle12,\beta\rangle$ 
\item $H_{\alpha+1}(O)=0$ if $a=13$
\item $H_{\alpha+1}(R)=H_{\alpha}(R)$ otherwise
\item $C_{0}(I)(\zeta)=I_{\zeta}$ for all $\zeta<\aleph$
\item $C_{\alpha+1}(W_{\beta})(\zeta)=0$ if $a=\langle1,\beta\rangle$
and $\zeta=H_{\alpha}(W_{\beta})$
\item $C_{\alpha+1}(O)(\zeta)=0$ if $a=5$ and $\zeta=H_{\alpha}(O)$
\item $C_{\alpha+1}(W_{\beta})(\zeta)=1$ if $a=\langle2,\beta\rangle$
and $\zeta=H_{\alpha}(W_{\beta})$
\item $C_{\alpha+1}(O)(\zeta)=1$ if $a=6$ and $\zeta=H_{\alpha}(O)$
\item $C_{\alpha+1}(\zeta)=C_{\alpha}(\zeta)$ otherwise for all $\zeta<\aleph$
\item $S_{\lambda}(R)=\lim\:\sup_{\alpha\rightarrow\lambda}S_{\alpha}(R)$
if $\lambda$ is a limit ordinal 
\item $H_{\lambda}(R)=\lim\:\sup_{\alpha\rightarrow\lambda}H_{\alpha}(R)$
if $\lambda$ is a limit ordinal
\item $C_{\lambda}(R)(\zeta)=\lim\:\sup_{\alpha\rightarrow\lambda}C_{\alpha}(R)(\zeta)$
if $\lambda$ is a limit ordinal
\end{itemize}
\begin{defn}

A serial $\langle\aleph,\beth,\daleth\rangle$-hypercomputer, for
cardinals $\aleph$, $\daleth$ and $\beth$ where $\beth\leq\daleth\leq\aleph$,
is a hypercomputer in which there are $\aleph$ many input, working
and output registers which each can store 0 or 1 and which supports
programs with $\beth$ states, with $\beth$ instructions (5-tuples),
and which supports a maximum of $o(\daleth)$ steps, where $o(\Gamma)$
is the least ordinal of cardinality $\Gamma$.
\end{defn}

\begin{defn}
A \emph{Turing machine} (see \cite{Tur1936}) is a $\langle<\aleph_{0},<\aleph_{0},<\aleph_{0}\rangle$-hyper\-computer
as it has $\aleph_{0}$ many registers but with only finitely many
registers addressed in the program, and each program having finitely
many states and instructions. Although a finite program may not halt,
a function is usually considered computable if there are \emph{< }$\aleph_{0}$
steps.
\end{defn}

\begin{defn}
\label{def:2.4}A \emph{parallel $\langle\aleph,\beth,\daleth\rangle$-hypercomputer},
for cardinals $\aleph$, $\daleth$ and $\beth$ where $\beth\leq\daleth\leq\aleph$,
is a\emph{ }hypercomputer that can store data in the registers and
process data from the registers in parallel. For the purposes of this
paper, such a parallel hypercomputer will comprise $\aleph-$many
serial \emph{$\langle\aleph,\beth,\daleth\rangle$}-hypercomputers
running independently in step but with the ability to use common read-only
input registers and the capability of writing outputs to a set of
registers through a second management program.\footnote{The general case is where the \emph{$\langle\aleph,\beth,\daleth\rangle$}-hypercomputers
are not independent of one another, but even in the general case the
dependency can be made explicit by taking the output of a parallel
\emph{$\langle\aleph,\beth,\daleth\rangle$}-hypercomputer as an input
to a serial \emph{$\langle\aleph,\beth,\daleth\rangle$}-hypercomputer
or to another parallel \emph{$\langle\aleph,\beth,\daleth\rangle$}-hypercomputer.} To be precise, there are $\aleph$ sets of registers $\{\langle I_{\alpha,\gamma},W_{\beta,\alpha,\gamma},O_{\alpha,\gamma}\rangle\}$,
 where $\gamma<\aleph$ is an index of the set of registers and in
fact an index of the overall parallel program, and the working and
output registers are disjoint, \emph{i.e.} $\bigcup_{\beta<\aleph,\alpha<\aleph}W_{\beta,\alpha,\gamma}\cap\bigcup_{\beta<\aleph,\alpha<\aleph}W_{\beta,\alpha,\delta}=\slashed{O}$
if $\gamma\neq\delta$ and $O_{\alpha,\gamma}\neq O_{\alpha,\delta}$
if $\gamma\neq\delta$. For each $\langle I_{\alpha,\gamma},W_{\beta<\aleph,\alpha,\gamma},O_{\alpha,\gamma}\rangle$
there is a program, $P_{\gamma}$, of length $\beth$ which runs disjoint
computations based on input registers $I_{\alpha,\gamma}$ for $\le\daleth$
steps and produces any output in $O_{\alpha,\gamma}$ for $\alpha<\aleph$.\footnote{Instructions in a parallel hypercomputer have the form \emph{$\langle Index\,of\,Serial\:hypercomputer,Current\,State,Current\,Set\,of\,Registers,\,Symbol,Action,\linebreak NextState\rangle$},
so that a program to copy input registers\emph{ $I_{\alpha,\gamma}$}
to working register\emph{ $W_{1,\alpha,\gamma}$ }(without the sequence
to move the program into the halting state) is\emph{ }$\langle\gamma,1,I,0,\langle1,1\rangle,1\rangle$,
$\langle\gamma,1,I,1,\langle2,1\rangle,1\rangle$, $\langle\gamma,1,I,0,4,1\rangle$,
$\langle\gamma,1,I,1,4,1\rangle$, $\langle\gamma,1,I,0,\langle8,1\rangle,1\rangle$,
where $\alpha$ is the current register in the input registers and
in the set\emph{ $W_{1,\gamma}$} in $\gamma-$th hypercomputer in
the parallel set. } There may be a separate management program \emph{M(Q)} that copies
the contents of all registers $O_{\alpha,\gamma}$ to the registers
in the initial state of a separate parallel\emph{ $\langle\aleph,\beth,\daleth\rangle$}-hypercomputer
and then runs a given program \emph{Q} ($=P_{\gamma}$), that in the
halting state contains the output of \emph{Q} (if any)\emph{.} For
ease of computation, it is assumed that parallel hypercomputers can
be chained, the output from one parallel hypercomputer being the input
to other parallel hypercomputers, and such a chain of hypercomputers
is also a parallel hypercomputer.\footnote{Allowing chains of parallel programs does not change the set of computable
functions, but can be useful in practice.}
\end{defn}

\begin{rem}
For infinite $\aleph$ a parallel \emph{$\langle\aleph,\beth,\daleth\rangle$-}hypercomputer
computes the same functions as a serial $\langle\aleph,\beth\times\aleph,\daleth\times\aleph\rangle$-hypercomputer,
\emph{i.e.} as a serial \emph{$\langle\aleph,\aleph,\aleph\rangle$}-hypercomputer,
as can be seen by noting that \emph{$\aleph$} computations can be
interleaved rather than being performed in parallel. For the same
reason, a parallel \emph{$\langle\aleph,\beth,\daleth\rangle$}-hypercomputer
is \emph{$\leq\aleph$} faster than a serial \emph{$\langle\aleph,\beth,\daleth\rangle$}-hypercomputer. 
\end{rem}

\section{Losslessly Compressed Sets}

It was mentioned in section 1 that it is possible to identify the
number of bits of information in a set $x$ (expressed as a binary
sequence that represents members of $x$ as well as $x$) as the least
length of the sequence which can be losslessly compressed from $x$.
The question arises how we express sets as binary sequences. While
it is possible to concatenate binary sequences representing members
of a set $X$ to represent $X$ as a binary sequence, here we will
fix an enumeration of $X\subseteq2^{\aleph}$, $\langle x_{\alpha}:\alpha<2^{\aleph}\rangle$
(which exists by the Axiom of Choice), and for any subset $Y\subseteq X$
form the binary \emph{$\aleph$-}sequence $\langle b_{\alpha}:(y_{\alpha}\in Y\rightarrow b_{\alpha}=1)\vee(y_{\alpha}\notin Y\rightarrow b_{\alpha}=0)\rangle$,
where the ordinal index of any member $y\in Y$ is taken from the
enumeration of $X$ (which includes all members of $Y$). This approach
has the advantages that all binary $2^{\aleph}$-sequences are represented,
and some sets where membership is easily decided are clearly compressible.
For example, $2^{\aleph}$ is represented as a $2^{\aleph}$-sequence
of 1s, while the empty set is represented as a $2^{\aleph}$-sequence
of 0s. Moreover, the representation of $2^{\aleph}-Y$ is formed from
the representation of $Y$ by swapping 0s for 1s. and \emph{vice versa}
\\
\\
A binary $\aleph$-sequence is \emph{losslessly compressible} if it
has an initial binary $<\aleph$-sequence followed by a terminal binary
$\aleph$-sequence which comprises $\aleph$ many repetitions of binary
$<\aleph$-sequences, and is \emph{losslessly incompressible }otherwise.
To see that this is a reasonable definition, note that it is possible
to create an $\aleph$-sequence by concatenating together with repetitions
a set of $<\aleph$-sequences of cardinality $\le\aleph$. If the
$\aleph$-sequence that results has period $<\aleph$, then the $\aleph$-sequence
can be treated as $\aleph$ many repetitions of binary $<\aleph$-sequences,
while if it has period $\aleph$ then it cannot be represented by
a $<\aleph$-sequence and thus is losslessly incompressible because,
if a set does not change its cardinality on being losslessly compressed,
it is treated as losslessly incompressible. The idea of a losslessly
compressible $\aleph$-sequence is that the sequence can be replaced
by a binary code for an initial $<\aleph$-sequence, a binary code
for the repeated pattern and a separate binary code for the number
of repetitions. The code for $\aleph$ repetitions can be set to 0
and for any other number of repetitions $r<\aleph$  the code can
be set to the cardinal of the ordinal $o(r)+1$. \\
\\
There is a clear link between the notion of information defined above
and Kolmogorov complexity (see \cite{Vitanyi97} for example). Recall
that Kolmogorov complexity of a set $X$ is the least length of a
computer program in a defined formal programming language which outputs
$X$. But while Kolmogorov complexity is a powerful and well-researched
approach to algorithmic complexity and to the study of randomness,
in this paper the focus will be on binary $\aleph$-sequences that
do not comprise $\aleph$ many repetitions of binary $<\aleph$-sequences
rather than sets which can be generated by a computable formula. The
primary reason for this choice is that the compressibility of a sequence
should only depend on patterns in the sequence and the sequence length
and not on a representation in a formal programming language\footnote{Using a universal Turing machine it can be shown that choice of programming
language imposes a constant overhead in terms of program length when
the program language is changed, see \cite{Fortnow2004}.}. Another difference with the approach of Kolmogorov complexity is
that Kolmogorov complexity minimizes program length, while here the
emphasis is on minimizing the number of steps in the computation of
a (serial hyper-)computer from a blank tape (or empty registers),
compressing the input data, running the program and decompressing
the output as necessary.\footnote{For decision problems no decompression is needed.}
It will also turn out that the program length is equal to the number
of steps in the information minimization principle below. In addition,
just like in Kolmogorov complexity, we make use of losslessly incompressible
binary sequences as a useful tool in proofs. To that end, we show
that there are sufficient losslessly incompressible sets of every
infinite cardinality. 
\begin{lem}
\label{lem:incompress}For every infinite cardinal $\aleph$, almost
all sets of cardinality $\aleph$ are not losslessly compressible
to sets of smaller cardinality.
\end{lem}

\begin{proof}
Firstly we recall that the number of bits in a set is always a cardinal
number. Proceed by an argument by cases on the cardinality of the
set: the infinite countable cardinal, $\aleph_{0}$; the infinite
successor cardinal case; and the infinite limit cardinal case. \\
\\
We first prove that almost all sets of cardinality $\aleph_{0}$ are
losslessly incompressible. We can note that there are $2^{\aleph_{0}}$
possible binary $\omega$-sequences, while there are only $\le\aleph_{0}$
$\omega$-sequences with a finite initial binary sequence and an independently
chosen terminal binary $\omega$-sequence comprising a repeated finite
binary sequence (since $(\sum_{Y\subseteq X,\left|Y\right|<\aleph_{0}}\left|Y\right|)\times(\sum_{Z\subseteq X,\left|Z\right|<\aleph_{0}}\left|Z\right|)=\aleph_{0}\times\aleph_{0}=\aleph_{0}$).
Hence almost all ($2^{\aleph_{0}}-\aleph_{0}=2^{\aleph_{0}}$) sets
of cardinality $\aleph_{0}$ (\emph{i.e}. sets expressible as a $\omega$-sequence)
are losslessly incompressible. \\
\\
When $\beth=\aleph+1$ is an infinite successor cardinal, then by
a counting argument there are $2^{\beth}$ binary $\beth$-sequences,
while there are $2^{\aleph}$ losslessly compressible $\beth$-sequences.
The latter can be shown by noting that there are $2^{\aleph}$ patterns
of length $\le\aleph$ in any terminal $\beth$-sequence and $2^{\aleph}$
initial binary $\aleph$-sequences, which are independent of one another,
\emph{i.e.} $2^{\aleph}\times2^{\aleph}=2^{\aleph}$ in total. Hence
almost all ($2^{\beth}-2^{\aleph}=2^{\beth})$ sets of cardinality
$\beth$ \emph{(i.e}. sets expressible as a $\beth$-sequence) are
losslessly incompressible.\\
\\
When $\beth$ is an infinite limit cardinal, by a counting argument
there are $2^{\beth}$ possible binary $\beth$-sequences, while there
are $\sum_{\alpha<\beth}2^{\alpha}$ losslessly compressible binary
$\beth$-sequences, where $\sum$ is the cardinal sum operator, because
by induction there are $2^{\aleph}$  losslessly compressible sets
for each infinite successor cardinal $\aleph+1$ (see the successor
cardinal case) and we can assume by hypothesis that there are $\sum_{\alpha<\gimel}2^{\alpha}$
losslessly compressible binary $\gimel$-sequences for limit cardinal
$\gimel<\beth$ . We can show that $\sum_{\alpha<\beth}2^{\alpha}<2^{\beth}$
by means of König's theorem. König's theorem states that $\sum_{i\in I}j_{i}<\prod_{i\in I}k_{i}$
for $I$ an index set, $j_{i}$ and $k_{i}$ are cardinals $j_{i}<k_{i}$,
and $\prod$ is the cardinal product function. If $I=\beth$, $j_{i}=2^{i}$
and $k_{i}=2^{\beth},$ then we have $\sum_{i<\beth}2^{i}<(2^{\beth})^{\beth}=2^{\beth}$
(see \cite{Jech2002} Theorem 5.16ii). Hence almost all ($2^{\beth}-\sum_{\alpha<\beth}2^{\alpha}=2^{\beth})$
sets of cardinality $\beth$ (\emph{i.e}. sets expressible as a $\beth$-sequence)
are losslessly incompressible. \\
\\
Since all three cases have been been established, the lemma follows.
\end{proof}

\section{The Generalised Continuum Hypothesis as an Information-Theoretic
Axiom}

In this section we prove a theorem that shows that GCH is an information-theoretic
axiom. First, however we define the notion of interleaved enumeration
for use in Theorem \ref{thm:GCH-is-equivalent}\emph{ et seq.} 
\begin{defn}
\label{def:An-interleaved-enumeration}An \emph{interleaved enumeration}
of two sets $U$ and $V$ is created by forming a new enumeration
$h$ from $f$ an enumeration function for $U$ and $g$ is an enumeration
function for $V$ as follows: $h_{\alpha}=f_{inf(\alpha)+fin(\alpha)/2}$
if ordinal $\alpha$ has a Cantor normal form\footnote{The Cantor normal form is a representation of any ordinal in the form
$\sum_{i=1}^{n<\omega}\omega^{b_{i}}\times c_{i}$ , where $c_{i}$
are positive integers and ordinals $b_{i}$ are such that $b_{i}>b_{j}$
and $b_{n}\ge0$ for $i<j$.} comprising an (possibly zero) infinite part $inf(\alpha),$ and an
even finite part $fin(\alpha)$ (including 0) and $h_{\alpha}=g_{inf(\alpha)+(fin(\alpha)+1)/2}$
if $fin(\alpha)$ is odd.
\end{defn}

\begin{thm}
\label{thm:GCH-is-equivalent}GCH is equivalent to\footnote{Strictly the inference from the information limitation principle to
GCH is probabilistic (true almost always) in cardinality terms rather
than logically necessary. } the assertion that the amount of information needed to decide the
relation $x\in X$ by an interleaved enumeration of $X$ and $2^{\aleph}-X$
is $<\aleph+1$, for any given binary $\aleph$-sequence x of length
at most cardinal $\aleph\ge\aleph_{0}$ and X has cardinality $\le2^{\aleph}$. 
\end{thm}

\begin{proof}
Assume that:
\begin{enumerate}
\item $\emptyset\subseteq X\subseteq2^{\aleph}$, 
\item \emph{$X$ }has cardinality $\aleph<c<2^{\aleph}$,
\item Any $x\in X$ is expressed as a binary sequence of length at most
cardinal $\aleph\ge\aleph_{0}$, and
\item The amount of information needed to decide the relation $x\in X$
by an interleaved enumeration of $X$ or $2^{\aleph}-X$ is $<\aleph+1$. 
\end{enumerate}
The proof is summarized in the tables below, where a $\checked$ means
that the option is possible and $\times$ means that the option is
impossible. \\

\begin{tabular}{|c|c|c|}
\hline 
 &
Enumerate $X$ &
Enumerate $2^{\aleph}-X$\tabularnewline
\hline 
\hline 
$x\in X$ &
$<c$ $\checked$ &
$2^{\aleph}$ $\times$\tabularnewline
\hline 
$x\notin X$ &
$c$ $\times$ &
$<2^{\aleph}$ $\checked$\tabularnewline
\hline 
\end{tabular}\\

\emph{Table 1: The number of steps to decide $x\in X$ by enumeration}\\
\\
\begin{tabular}{|c|c|c|c|}
\hline 
$<c$ &
Proof Ref. &
$c$ &
Proof Ref.\tabularnewline
\hline 
\hline 
$\aleph+1<c$ $\times$ &
1 &
$\aleph+1<c$ $\times$ &
4\tabularnewline
\hline 
$\aleph+1=c$ $\checked$ &
2 &
$\aleph+1=c$ $\times$ &
5\tabularnewline
\hline 
$\aleph+1>c$ $\times$ &
3 &
$\aleph+1>c$ $\times$ &
3\tabularnewline
\hline 
\end{tabular}\\
\\
\\
\begin{tabular}{|c|c|c|c|}
\hline 
$<2^{\aleph}$ &
Proof Ref. &
$2^{\aleph}$ &
Proof Ref.\tabularnewline
\hline 
\hline 
\textbf{$\aleph+1<2^{\aleph}$ }$\times$ &
1 &
$c<2^{\aleph}$ $\times$ &
8\tabularnewline
\hline 
\textbf{$\aleph+1=2^{\aleph}$ }$\checked$ &
6 &
$c<2^{\aleph}$ $\times$ &
8\tabularnewline
\hline 
\textbf{$\aleph+1>2^{\aleph}$ }$\times$ &
7 &
$c<2^{\aleph}$ $\times$ &
8\tabularnewline
\hline 
\end{tabular}\\
\\
\emph{Table 2: The possible cardinal relationships for the number
of steps in Table 1 and proof references}\\
\emph{}\\
Proof references:\\
\\
1. $x\in X$ would almost always be decided in $\ge\aleph+1$ bits
for a given enumeration of $X$, contradicting assumption d). \\
2. $\aleph+1=c$ is consistent with assumption d), as $x\in X$ would
be decided in $<c=\aleph+1$ steps by enumeration.\\
3. $\aleph+1>c$ contradicts assumption b) $\aleph<c$, as there would
be a cardinal strictly between $\aleph$ and $\aleph+1$.\\
4. $x\in X$ would almost always be decided in $>\aleph+1$ bits for
a given enumeration of $X$, contradicting assumption d). \\
5. $\aleph+1=c$ implies that $\aleph+1$ bits are needed to decide
$x\in X$ by enumerating all of $X$, which contradicts assumption
d).\\
6. \textbf{$\aleph+1=2^{\aleph}$ }is consistent with assumption d),
as $x\in X$ would be decided in $<2^{\aleph}=\aleph+1$ steps by
enumeration.\\
7. \textbf{$\aleph+1>2^{\aleph}$ }contradicts Cantor's theorem tha\textbf{t
$\aleph+1\le2^{\aleph}$.}\\
8. $c<\left|2^{\aleph}-X\right|=2^{\aleph}$ and therefore $x\in X$
could always be decided in $<2^{\aleph}$ steps by enumeration.\\

We can conclude that if $x\in X$ then $c=\aleph+1$ and if $x\notin X$
then $\aleph+1=2^{\aleph}$. Using predicate logic\footnote{Existential elimination: for example, assume $(\exists x)(x\in X)$
and $(\forall x)(x\in X\rightarrow c=\aleph+1)$, then if $c\neq\aleph+1$
then by contraposition $(\forall x)(x\notin X)$ and hence $\neg(\exists x)(x\in X)$,
contradiction; hence $c=\aleph+1$.} we can conclude $(\exists x)(x\in X)\rightarrow c=\aleph+1$ and
$(\exists x)(x\in2^{\aleph}-X)\rightarrow\aleph+1=2^{\aleph}$. Since
both \emph{X} and $2^{\aleph}-X$ are not empty we can conclude that
$c=\aleph+1=2^{\aleph}$, which contradicts assumption b) that $c<2^{\aleph}$.
GCH then follows.\\
\\
Conversely, assume GCH. Then if $x\in X$ then by GCH $x$ will be
enumerated in $<\left|X\right|\le2^{\aleph}=\aleph+1$ steps. While
if $x\notin X$ then $x$ will be enumerated in $<\left|2^{\aleph}-X\right|=2^{\aleph}=\aleph+1$
steps. In either case then $x\in X$ can be decided by enumeration
in $<\aleph+1$ steps, \emph{i.e.} in $<\aleph+1$ bits. 
\end{proof}
\begin{rem}
In the proof above of Theorem \ref{thm:GCH-is-equivalent} there is
an assumption that an interleaved enumeration need not take more than
$\aleph$ bits to decide $x\in X$, and the proof of GCH by contradiction
is only valid if $x$ is sufficiently generic (a random variable)
to be in the bulk of an enumeration of $X$, \emph{i.e.} at $\ge\aleph+1$
steps from the start in proof references 1, 4 and 5 above. We can
do this by using the axiom of choice to enumerate $X$ such that each
$y\in X$ is decided in $<\aleph+1$ bits by interleaved enumeration
(which is all $y\in2^{\aleph}$) and choose $x$ to be in the bulk
of an enumeration of $X$. 
\end{rem}

\section{\label{sec:An-Information-Minimization}An Information Minimization
Principle}

This information minimization principle is an expression of the fact
that all sets and all membership relations can be hypercomputed and
that a set and a relation contain a certain number of bits of information,
and it does not matter how those bits are enumerated, as some enumeration
of this number of bits will define the set and decide the truth of
the relation for particular sets. We could in fact define a set $X$
of cardinality $\le2^{\aleph}$ as a set of sets $x$ that can be
defined in $\le\aleph$ bits by enumeration and the membership relation
between $x$ and $X$ (see Theorem \ref{thm:minimum-info}) can be
decided in $\le\aleph$ bits by enumeration. \\
\\
This may seem in conflict with the finite case, but membership of
a finite set of $2^{n}$ members for $n\ge1$ (which can be taken
to be natural numbers or binary sequences representing natural numbers)
can be decided in $\le n+1$ bits by using a binary search algorithm
if $X$ and the complement of $X$ are ordered in ascending order,
say $X=\{x(i):1\le i\le2^{n}\}$ and $x(i)=-1$ if $x(i)$ is not
defined. Then to decide whether $x\in X$, follow the algorithm in
the following pseudo-code, where all variables are natural numbers.\\
\\
\noindent\fbox{\begin{minipage}[t]{1\columnwidth - 2\fboxsep - 2\fboxrule}%
Set $left=0$

Set $right:=2^{n}$

Loop while $(left\le right)$

$mid:=(left+right)/2$

if $x(mid)\ge0$ then:
\begin{itemize}
\item if $x(mid)=x$ then return True
\item if $x(mid)<x$ then $left=mid$+1
\item if $x(mid)>x$ then $right:=mid$-1
\end{itemize}
End loop

return False%
\end{minipage}}\\
\\
\\
This program runs for $\le n+1$ steps in terms of the number of members
of $X$ enumerated. Of course the binary search could also be applied
to the complement of $X$, but the run time is again $\le n+1$ steps
in an enumeration. While the efficient enumeration of $X$ or the
complement to decide $x\in X$ relies on specific linear orderings
of $X$, the search process defines a binary expansion (whether the
midpoint is to the ``left'' or ``right'' of $x$ in the ordering)
of any $x\in2^{n}$ with a final member of the sequence representing
the decision whether $x\in X$ or not. In fact the binary expansion
of ``left'' and ``right'' labels mutually defines the sequence
of midpoints. This is suggestive of the approach in Theorem \ref{thm:GCH-is-equivalent}
that an efficient enumeration represents $x\in X$ as a binary $\aleph$-sequence
representing $x$ followed by a decision whether $x\in X$. It lends
support to the view that an efficient enumeration of $X\subseteq2^{\aleph}$
is always representable as a binary $\aleph$-sequence representing
$x$ followed by a decision whether $x\in X$. More generally, it
is also possible to use a midpoint construction where $X$ is a dense
subset of a closed interval in the standard topology of the real line,
say $[0,1]$, by choosing the midpoint of the interval if the midpoint
is a member of the interval or choosing a member of the set near the
midpoint (using the Axiom of Choice) otherwise. Then the $\omega$-sequence
of near-midpoints will converge to the point $x$ in the interval
(unless the near-midpoint algorithm chooses $x$ at some finite stage
in the enumeration), which may or may not be a member of $X$. \\
\\
By the definition of the number of bits of information, for every
set $X\subseteq2^{\aleph}$ there is a losslessly compressed set $Y$
(\emph{i.e.} cardinality $\left|Y\right|\le\left|X\right|)$ that
contains the same information as $X$. Let us assume that we can well-order
a binary $2^{\aleph}$-sequence in a monotonic way with the constant
sequence with the smaller cardinality as the initial sequence. This
is possible by choosing members of the sequence with value 0 and building
a sequence and doing the same for members of the sequence with value
1, and then concatenating them with the smallest set first. Otherwise
if the sequences have equal length of $2^{\aleph}$ a binary $2^{\aleph}$$\times2$-sequence
will be needed. Then we see that the maximum lossless compression
occurs when one of the constant sequences is empty, and in general
lossless compressibility will depend on the cardinality of the smaller
constant sequence. But is this the minimum amount of bits needed to
decide $x\in X$? The answer in general is ``no'' because each $x$
has a representation as a binary $\aleph$-sequence and it possible
to add an extra bit to every binary $\aleph$-sequence to indicate
whether $x\in X$ or not. Lossless compressibility adds complication
to computation of the minimum steps in the computation of $x\in X$
because in general the index of $x$ in an arbitrary binary $2^{\aleph}$-sequence
will need to be represented as an ordinal $<2^{\aleph}$ but when
the repeated pattern is a constant value then $x\in X$ can be determined
in a number of bits $\le\beth$, where $\beth$ is the length of the
initial sequence before the repeated pattern. If, however, we consider
only losslessly incompressible sets $X$ we can state a principle
of information minimization as follows:\footnote{This view does not contradict the speed up theorems in formal axiomatic
systems, see \cite{Buss1994}, because axiomatic systems constrain
the proof method to a finite sequence of computation steps, albeit
from a number of different axioms that depend on the axiom system.} \\
\\
\emph{}%
\noindent\fbox{\begin{minipage}[t]{1\columnwidth - 2\fboxsep - 2\fboxrule}%
\emph{Principle of Information Minimization: }For all losslessly incompressible
sets $X\subseteq2^{\aleph}$ and $x\in2^{\aleph}$ and for all relations
$x\in X$ there is a minimum amount of information $\mu$ such that
if a $\langle\nu,\nu,\nu\rangle$ -hypercomputer\footnote{Numbers of registers and states that are greater than the length of
the computation are not used; hence the number of states and registers
are set equal to the length of the computations.} can decide $x\in X$ in $\le\nu$ steps by any enumeration of $X$
and $2^{\aleph}-X$, it follows that a $\langle\mu,\mu,\mu\rangle$
-hypercomputer can decide $x\in X$ in $\mu\le\nu$ steps by an interleaved
enumeration of $X$ and $2^{\aleph}-X$. %
\end{minipage}}\\
\\
The argument for the Principle of Information Minimization is that
an enumeration that locates $x$ in an interleaved way in $X$ and
$2^{\aleph}-X$ (which is efficient for infinite sets $X$) should
take a number of steps no more than the number of bits of information
in $x\in X$. It should be noted that if $X$ is losslessly incompressible
then so is $2^{\aleph}-X$ as any pattern in a binary $2^{\aleph}$-sequence
representing $X$ will correspond to a bit-flipped pattern in a representation
of $2^{\aleph}-X$. As further motivation for this argument, we can,
as noted above, regard the shortest enumeration of a member $x$ of
a set $X$ as an optimal search algorithm for $x$. That is to say,
each successive bit of $x$ corresponds to a choice of (nested) intervals
in a linear order of $X$. Each interval can be represented by a member
of the interval, and after the number of bits equal to the length
of $x$, $\aleph$, $x$ will be definitely be located or not, i.e.
$x\in X$ will be decided by enumeration. This motivation will not
be pursued further in this paper because it needs the development
of topological arguments to explain the idea more fully. \\
, \\
In Theorem \ref{thm:minimum-info}  below we show (highly non-constructively)
that the number of bits of information in the relation $x\in X$ for
$X\subseteq2^{\alpha}$ is $<\aleph+1$. Let us take $X\subset2^{\aleph}$
to be a losslessly incompressible set of cardinality $2^{\aleph}$
which is \emph{entangled} in losslessly incompressible set $2^{\aleph}-X$,
\emph{i.e.} each $\aleph$-sequence in $X$ is covered by $\aleph$-sequences
in $2^{\aleph}-X$ and \emph{vice vers}a.\footnote{We can also say that both $X$ and $2^{\aleph}-X$ are dense in $2^{\aleph}$.}
The most interesting case\footnote{The other cases are dealt with in Corollary \ref{cor:GCH}. }
is when both $X$ and $2^{\aleph}-X$ have cardinality $2^{\aleph}$.
$X$ can be constructed by the Axiom of Choice, making sure that for
each initial $<\aleph$-sequence, $s$, one $\aleph$-sequence $x$
that has $s$ as an initial $<\aleph$-sequence is selected to be
put in $X$, and one $\aleph$-sequence $y\ne x$ that has $s$ as
an initial $<\aleph$-sequence is selected to be put in $2^{\aleph}-X$;
and dividing other members of $2^{\aleph}$ equally among $X$ and
$2^{\aleph}-X$ (by well-ordering $2^{\aleph}-S$, where $S$ is the
set of $\aleph$-sequences already selected, and alternately putting
members of the well-order in $X$ and $2^{\aleph}-X$, putting limit
ordinal members in $X$ for definiteness, since the number of successor
ordinals is the same as the number of limit ordinals $<2^{\aleph})$.
It is shown in Lemma \ref{lem:incompress} that almost all sets of
infinite cardinality are not losslessly compressible in terms of number
of bits of information, so we can choose two incompressible sets of
cardinality $2^{\aleph}$ (as the constraint of $X$ and $2^{\aleph}-X$
each containing a dense subset of cardinality $\aleph$ does not affect
the choice of other members of $X$ and $2^{\aleph}-X$). Corollary
 \ref{cor:GCH} shows that GCH follows from Theorem \ref{thm:GCH-is-equivalent}
for $X$ and $2^{\aleph}-X$ incompressible.
\begin{thm}
\label{thm:minimum-info}A universal hypercomputer computes the minimal
amount of information needed to decide the relation $x\in X$, where
$x$ is any binary sequence of length at most cardinal $\aleph\ge\aleph_{0}$
and X has cardinality $\le2^{\aleph}$, in $V$ as $\aleph$.
\end{thm}

\begin{proof}
For a set $X$ that consists of binary $\aleph$-sequences, associate
to every binary $\aleph$-sequence $x\in2^{\aleph}$ a $o(\aleph)+1$-sequence
$x\cup\{\langle o(\aleph+1),1\rangle\}$ if $x\in X$ and $x\cup\{\langle o(\aleph+1),0\rangle\}$
if $x\notin X$, where $\langle a,b\rangle=\{a,\{a,b\}\}$. Call the
associated set $2^{\aleph(}(X)$. Properties of sets can be recovered
from the associated sets, \emph{e.g.} $x\in X$ if $x\cup\{\langle o(\aleph+1),1\rangle\}\in2^{\aleph}(X)$,
$X\subseteq Y$ if $(\forall(x\cup\{\langle o(\aleph+1),1\rangle\})\in2^{\aleph}(X))(x\cup\{\langle o(\aleph+1),1\rangle\}\in2^{\aleph}(Y))$.
Associated sets are sets where membership is always decided, which
is true of membership computed by a universal hypercomputer (see Theorem
\ref{rthm:serial-hyper})\footnote{It is possible to take a topological approach to the hypercomputation
of $x\in X$. The set $X$ can be given a topology where basic open
sets are sets of $\aleph$-sequences that extend some initial $<\aleph$-sequence.
It can be seen that basic open sets are closed as well as open (because
they have their own limit points and no limit points belonging to
their complement in $X$). The intersection of basic clopen (closed
and open) sets that are neighbourhoods of $x$ have intersection $x$
if $x\in X$ and is empty otherwise. After $\aleph$ steps a hypercomputer
can decide whether $x\in X$ or not. Replacement of sets by associated
sets is a clearer hypercomputational approach to deciding set membership
than a topological approach.}. If then a set $X$ is identified with its associated set $2^{\aleph(}(X)$,\footnote{Associated sets obey the standard rules of intersection, union and
complement but only functions from one set to another that are allowed
are those that preserve the $o(\aleph)+1$-th member of the binary
sequence representing set $x$ in the domain of the function. However,
it is true that every set that exists in the Von Neumann universe
of sets, $V$, has an associated set, because $2^{\aleph(}(X)$ can
always be hypercomputed from $X.$ } then any $x\in X$ can be decided in steps of cardinality $\le\aleph$,\emph{
i.e.} in $<\aleph+1$ bits, by enumerating the $o(\aleph)+1$-sequence
$y\in2^{\aleph}(X)$ corresponding to $x$, and checking its $o(\aleph)+1$-th
member. $x\in X$ cannot be decided in $<\aleph$ steps in general
because $x$ requires $\aleph$ bits to be specified if $x$ is a
losslessly incompressible binary $\aleph$-sequence (which always
exist for infinite $\aleph$ by Lemma \ref{lem:incompress}). 
\end{proof}
\begin{cor}
\label{cor:GCH}GCH is computed as true in $V$if the Information
Minimization Principle holds.
\end{cor}

\begin{proof}
By Theorem \ref{thm:minimum-info}  $\aleph$ is the minimum number
of bits needed to decide $x\in X$. Choose $X\subset2^{\aleph}$ and
$2^{\aleph}-X$ to be losslessly incompressible sets of cardinality
$2^{\aleph}$ (see Remark 8). Since $X$ and $2^{\aleph}-X$ are losslessly
incompressible sets, it follows that we can apply the Information
Minimization Principle. Then, since a $\langle2^{\aleph},2^{\aleph},2^{\aleph}\rangle$
-hypercomputer can decide $x\in X$ by interleaved enumeration of
$X$ and $2^{\aleph}-X$ in $<2^{\aleph}$ bits, it follows from the
Information Minimization Principle that a $\langle\aleph+1,\aleph+1,\aleph+1\rangle$
-hypercomputer can decide $x\in X$ by interleaved enumeration\footnote{The program for interleaved enumeration will loop through members
of $X$ and $2^{\aleph}-X$  in a specific order and exit and return
true when it matches a specific $\aleph$-sequence $x$ (which it
always does in $<2^{\aleph}$ steps). It is identical to the program
for $(\exists y)R(y)\wedge\exists y)S(y)$ in Theorem \ref{rthm:serial-hyper}
below with $R(y)=(x=y\wedge y\in X).$and $S(y)=(x=y\wedge y\in2^{\aleph}-X)$,
although we only need to load in $\aleph+1$ members of $X$ and $2^{\aleph}-X$
in light of the Information Minimization Principle.} of $X$ and $2^{\aleph}-X$ in (any ordinal of cardinality) $\aleph$
steps,\emph{ i.e}. the number of steps is $<\aleph+1$ bits. Hence
$2^{\aleph}=\aleph+1$ follows directly from Theorem \ref{thm:GCH-is-equivalent};
or we can note that we have $x\in X$ if and only $x$ is in an interleaved
enumeration of $X$ and $2^{\aleph}-X$ in $<2^{\aleph}$ steps (since
an interleaved enumeration can be created from enumerations of $X$
and $2^{\aleph}-X$, see Definition \ref{def:An-interleaved-enumeration})
only if $x$ is in an interleaved enumeration of $X$ and $2^{\aleph}-X$
in $<\aleph+1$ steps for losslessly incompressible $X$ and $2^{\aleph}-X$.
It follows that $2^{\aleph}\le\aleph+1$, and $\aleph+1\le2^{\aleph}$
by Cantor's theorem. Other cases are where $X$ and $2^{\aleph}-X$
are losslessly incompressible sets and one has cardinality $\le\aleph$
(including being empty or being countable); and where $X$ and $2^{\aleph}-X$
are incompressible sets and one of $X$ or $2^{\aleph}-X$ and has
cardinality $\aleph<c<2^{\aleph}$. The former case shows that $x\in X$
can be decided in $\le\aleph$ steps, which is consistent with Theorem
\ref{thm:GCH-is-equivalent}. The latter case is shown by Theorem
\ref{thm:GCH-is-equivalent} to be impossible (since $x\in X$ can
be decided in $<c$ or $<2^{\aleph}$ steps, leading to $c=2^{\aleph}=\aleph+1$).
Hence we have shown GCH is computed as true in $V$ based on Theorem
\ref{thm:minimum-info}.
\end{proof}
\begin{rem}
The result in Theorem \ref{thm:minimum-info} is highly non-constructive,
and relies on the Information Minimization Principle and on there
existing a set $Y$ which corresponds to set $X\subseteq2^{\aleph}$
such that $Y$ is a set of $o(\aleph)+1$-sequences which computes
the decision problem for every $x\in X$ and appends the results to
the $\aleph$-sequence for $x$ in $2^{\aleph}$. This set $Y$, or
$2^{\aleph}(X)$ as it was called in Theorem \ref{thm:minimum-info},
is not computable in general by a finite computer, but needs a (universal)
hypercomputer. It is possible, as noted above, to reject this view
on the grounds of its computational or ontological assumptions (that
every set is computable and every relation decidable). It is also
possible to substitute other bounds on the decision problem for $x\in X$,
such as linking sets to formulas of fixed bounded quantifier complexity,;
but those bounds of course would also need motivation. It is also
worth noting that Theorem \ref{thm:minimum-info} also leads to a
very nice structure for, for example, the real numbers. Two entangled
uncountable sets of real numbers are either one countable set entangled
with an uncountable set of real numbers (\emph{viz.} a continuum)
or two entangled continua.\footnote{Of course the topological properties of the two entangled continua
may be different, for example a Cantor set and an open dense continuum.} 
\end{rem}

\section{Results about the Universal hypercomputer}
\begin{thm}
\label{rthm:serial-hyper} A serial $\langle2^{\aleph},2^{\aleph},2^{\aleph}\rangle$-hypercomputer
can compute a) the truth of first-order propositions with quantification
over sets that require $\leq\aleph$ bits of information to define,
b) the truth of first-order propositions like a) but with the addition
of allowing set membership of sets that require $\leq2^{\aleph}$
bits of information to define, and c) a serial $\langle2^{2^{\aleph}},2{}^{2^{\aleph}},2{}^{2^{\aleph}}\rangle$-hypercomputer
can compute the truth of second-order propositions about sets that
require $\leq\aleph$ bits of information to define.
\end{thm}

\begin{proof}
a) To start, the truth of recursive relations involving finitely many
sets that require $\leq\aleph$ bits of information to define (including
the standard logical operators $\wedge$, $\vee,$$\rightarrow$,
$\leftrightarrow$ and $\neg$) can be decided by a program with finitely
many instructions in $\leq\aleph$ steps because the recursive relation
generates a finite program and\emph{ }$\leq\aleph$ steps are needed,
one for each bit.\emph{ }Then to decide\emph{ $(\forall x)R(x)$ }for\emph{
x }a set that requires $\leq\aleph$ bits of information to define
and\emph{ R }recursive, loop through the set of all sets that require
$\leq\aleph$ bits of information to define, run the program for \emph{R(x)}
in disjoint register sets in series, and then copy the results (0
or 1, \emph{i.e.} false or true) to another disjoint set of registers,
the computation having $2^{\aleph}$ steps\footnote{Any set that requires $\leq\aleph$ bits of information to define
can be either be a member or not a member of the set of such sets;
hence the cardinality of the set of all sets that require $\leq\aleph$
bits of information to define, \emph{X} say, is the same as the set
of all functions $\aleph\rightarrow2$, \emph{i.e.} $2^{\aleph}$.
Hence the total number of steps to loop through every member of $X$
is $\aleph\times2^{\aleph}=2^{\aleph}$. }. To ``loop through'' the quantification domain, coding can be used
to detect in finitely many instructions which registers have been
accessed by the program,\footnote{If a sequence $\langle a_{1},a_{2},\cdots,a_{i<\aleph},\cdots\rangle$
of length $2^{\aleph}$, where $a_{i}$ is a member of the quantification
domain and a binary sequence of length $<\aleph+1$, is coded as $\langle a_{1},1,a_{2},1,\cdots,1,a_{i<2^{\aleph}},1,\cdots\rangle$,
by placing a 1 marker after every successor and limit member of the
sequence, then the \emph{1} can be replaced with \emph{0} if the previous
register has been accessed by the program. The program can proceed
until it finds a register succeeded by a \emph{1}.} and the least unaccessed member of the set can be accessed next\footnote{Looping requires one new state, which acts as a label for the start
of the loop and which which is the next state for instructions in
the loop after the program for $R(x)$ has run.}. To create and load all sets that require $\leq\aleph$ bits of information
to define requires a program of length $\leq2^{\aleph}$ because there
are $\leq2^{\aleph}$ such sets to be computed, each requiring $\leq\aleph$
instructions. The conjunction (``and'') of the truth values of \emph{R(x)}
is then computed by a finite program (see footnote \ref{fn:9} for
the outline of a finite program to compute the truth value of a conjunction),
and \emph{$(\forall x)R(x)$ }is true if and only if the conjunction
has value \emph{1} (true). \emph{$(\exists x)R(x)$} can be decided
similarly using disjunctions (``or'') rather than conjunctions.
By induction on quantifier complexity the truth of any first-order
proposition about sets that require $\leq\aleph$ bits of information
to define (with a recursive quantifier free formula) can be decided
by a\emph{ }$\langle2^{\aleph},<\aleph_{0},2^{\aleph}\rangle$-hypercomputer
given a set of sets that require $\leq\aleph$ bits of information
to define. If the loading of the input is included, a serial $\langle2^{\aleph},2^{\aleph},2^{\aleph}\rangle$-hypercomputer
suffices to compute the truth of any first-order quantified proposition
about sets that require $\leq\aleph$ bits of information to define.\\
\\
b) To show that a first-order quantified proposition with quantification
over sets that require $\leq\aleph$ bits of information to define
and with the addition of specific sets that require $\leq2^{\aleph}$
bits of information to define can also be computed by a serial $\langle2^{\aleph},2^{\aleph},2^{\aleph}\rangle$-hypercomputer,
we note that a serial $\langle2^{\aleph},2^{\aleph},2^{\aleph}\rangle$-hypercomputer
can compute any set that require $\leq2^{\aleph}$ bits of information
to define by starting with a blank tape (\emph{i.e.} all \emph{0}s)
and running a program of length $2^{\aleph}$ to write a value (\emph{0
}or \emph{1}) to each register. Membership of a set, $x\in X$, where
each \emph{x} must take $\leq\aleph$ bits to define to be consistent
with a),\footnote{$x\in2^{\aleph}$ as x takes $\leq\aleph$ bits to define.}
can therefore be computed by a serial $\langle2^{\aleph},2^{\aleph},2^{\aleph}\rangle$-hypercomputer
by looping through the set \emph{X} with current value $y\in X$ and
checking whether $y=x$. The inductive argument in a) above can then
be applied to show that a serial $\langle2^{\aleph},2^{\aleph},2^{\aleph}\rangle$-hypercomputer
can compute the truth of any first-order proposition with quantification
over sets that require $\leq\aleph$ bits of information to define
and which have set membership of sets that require $\leq2^{\aleph}$
bits of information to define.\\
\\
c) The truth of a second-order proposition of set theory with quantification
over sets that require $\leq2^{\aleph}$ bits of information and sets
of sets that require $\leq\aleph$ bits of information can be decided
by ``looping through'' every set of sets that require $\leq\aleph$
bits of information,\footnote{Note that a marker such as $\langle1,0,1\rangle$ can be added to
each set of sets that require $\leq\aleph$ bits of information to
define in the sequence of registers.} which requires $2^{2^{\aleph}}$ registers and $2^{2^{\aleph}}$
steps with a finite program and which depends on $2^{2^{\aleph}}$
instructions to create and ``load'' the data, \emph{i.e.} the set
of sets of sets that require $\leq\aleph$ bits of information.
\begin{thm}
A parallel \emph{$\langle2^{\aleph},\aleph,\aleph\rangle$}-hypercomputer
can compute a) the truth of first-order propositions with quantification
over sets that require $\leq\aleph$ bits of information to define,
b) the truth of first-order propositions like a) but with the addition
of allowing set membership of sets that require $\leq2^{\aleph}$
bits of information to define, and c) a parallel $\langle2^{2^{\aleph}},\aleph,\aleph\rangle$-hypercomputer
can compute the truth of second-order propositions about sets that
require $\leq\aleph$ bits of information to define.
\end{thm}

a) Note that a parallel \emph{$\langle2^{\aleph},\aleph,\aleph\rangle$}-hypercomputer
can write $2^{\aleph}$ sets that require $\leq\aleph$ bits of information
to define into the registers in parallel. Proceed by induction with
the hypothesis that a parallel \emph{$\langle2^{\aleph},\aleph,\aleph\rangle$}-hypercomputer
can compute the truth of first-order quantified propositions of sets
that require $\leq\aleph$ bits of information to define, noting that
for the basis case of a recursive relationship between finitely many
sets that require $\leq\aleph$ bits of information to define it takes
$\leq\aleph$ instructions and $\leq\aleph$ steps to write finitely
many sets that require $\leq\aleph$ bits of information to define
to a set of registers and then finitely many instructions and $\leq\aleph$
steps to compute the recursive relationship for those sets. For the
induction step, note that for $(\forall x)R(x)$ or $(\exists x)R(x)$,
$2^{\aleph}$ sets that require $\leq\aleph$ bits of information
to define can be loaded by a parallel\emph{ $\langle2^{\aleph},\aleph,\aleph\rangle$}-hypercomputer
across $2^{\aleph}$ disjoint sets of $2^{\aleph}$ registers and
the quantification can be parallelised by running a (finite) program
for deciding\emph{ R(x)} in parallel in $\aleph$ steps, for $(\forall x)R(x)$
writing\emph{ 1 }to an output register of the management program initially
and then writing\emph{ 0} to the output register if any of the \emph{R(x)}
computes as false, while for $(\exists x)R(x)$ writing\emph{ 0 }to
an output register initially and then writing\emph{ 1} to the output
register if any of the \emph{R(x)} computes as true. \\
\\
b) If we add propositions involving membership of $\leq2^{\aleph}$
specific sets, assumed for consistency with a) to consist of members
which have $\leq\aleph$ bits to define, then to write a specific
set requires a parallel \emph{$\langle2^{\aleph},\aleph,\aleph\rangle$}-hypercomputer
if each disjoint set of $2^{\aleph}$ registers contains one set that
requires $\leq\aleph$ bits of information to define.\footnote{It is assumed that the $\leq\aleph$ bits are presented serially and
cannot be parallelised, for example by a recursive relationship.} Testing membership of a specific set of sets that require $\leq\aleph$
bits of information to define, \emph{r}, requires matching \emph{r}
against $2^{\aleph}$ disjoint sets of registers which contain one
set that requires $\leq\aleph$ bits of information to define, $s_{\alpha<2^{\aleph}}$,
which can be done in parallel with a finite program in $\aleph$ steps
as follows. Use \emph{r} and $s_{\alpha}$ from the input registers
and create a set of working registers, \emph{$D_{\alpha<2^{\aleph}}$},
with one register each, written $W_{\alpha}$, in $1$ step and with
a finite program writing 1 to each $W_{\alpha}$ in parallel. For
\emph{r} and each $s_{\alpha}$, for ordinal \emph{$\beta<\aleph$}
perform the operation $(r)_{\beta}\leftrightarrow(s_{\alpha})_{\beta}$\footnote{That is $((r)_{\beta}\wedge(s_{\alpha})_{\beta})\vee((\neg r)_{\beta}\wedge(\neg s_{\alpha})_{\beta})$.}
in parallel, which returns \emph{1} if $(r)_{\beta}=(s_{\alpha})_{\beta}$
and 0 otherwise; and if the result is \emph{0} write \emph{0} to $W_{\alpha}$
and then halt the program; otherwise write \emph{1} to $W_{\alpha}$
and then move right one register along\emph{ r} and $s_{\alpha}$
to $(r)_{\beta+1}$ and $(s_{\alpha})_{\beta+1}$. At limit ordinals
$\lambda$, proceed as normal by performing the operation $(r)_{\lambda}\leftrightarrow(s_{\alpha})_{\lambda}$.
\footnote{To implement the pseudo-code as a program, it is possible to use a
hypercomputer with three ordinary states, 2,3,4, an initial state,
\emph{1}, a halting state, 5, with the following instructions, assuming
that the program starts in state \emph{1}, that two sets that require
$\leq\aleph$ bits of information to define are for simplicity stored
in $W_{1,\alpha<\aleph,\gamma}$ and $W_{2,\alpha<\aleph,\gamma}$,
the result of bit-wise comparison of the sets that require $\leq\aleph$
bits of information to define is stored in $W_{3,1}$. A suitable
program is $\langle\gamma,1,W_{1},0,\langle2,3\rangle,4\rangle$,
$\langle\gamma,1,W_{1},1,\langle2,3\rangle,4\rangle$, $\langle\gamma,4,W_{1},0,\langle8,1\rangle,3\rangle$,
$\langle\gamma,4,W_{1},1,\langle8,1\rangle,2\rangle$, $\langle\gamma,3,W_{2},1,\langle1,3\rangle,5\rangle$,
$\langle\gamma,2,W_{2},0,\langle1,3\rangle,5\rangle$, $\langle\gamma,3,W_{2},0,\langle8,1\rangle,4\rangle$,
$\langle\gamma,2,W_{2},1,\langle8,1\rangle,4\rangle,\langle\gamma,1,W_{0},0,\langle12,0\rangle,1\rangle,\,\langle\gamma,1,W_{0},1,\langle12,0\rangle,1\rangle,\,\langle\gamma,1,W_{0},1,0,5\rangle$.
The reason that the state with the main loop is the highest ordinary
state is 4 is to allow the program to start in the main loop at limit
ordinals. It can be seen that the program will either halt in state
5 with output 0 or in state 4 with output 1 when the computation runs
to completion (\emph{i.e.} at step $o(\aleph)$).} \\
c) Each of a maximum of $2^{2^{^{\aleph}}}$ sets of sets that require
$\leq2^{\aleph}$ bits of information to define can be represented
as specific sets when computing the truth of first-order quantified
propositions involving such sets. Put more formally, since a parallel
\emph{$\langle2^{\aleph},\aleph,\aleph\rangle$}-hypercomputer can
compute the truth of a first-order quantified proposition with quantification
over sets of sets that require $\leq\aleph$ bits of information to
define with the addition of membership of specific sets that require
$\leq2^{\aleph}$ bits of information to define, if \emph{R(X)}, for
\emph{X} a set of sets that require $\leq\aleph$ bits of information
to define, is a formula of set theory with free variable \emph{X},
then $(\forall X)R(X)$ can be computed in parallel across $2^{2^{^{\aleph}}}$
disjoint sets of $2^{\aleph}$ registers by writing \emph{1 }to an
output register of the management program initially and then writing
\emph{0} if any of \emph{R(X)} is false; and for \emph{$(\exists X)R(X)$}
by writing \emph{0 }to an output register initially and then writing
\emph{1} if any of \emph{R(X)} is true. By induction on quantifier
complexity of a second-order predicate \emph{A(X)}, since the parallel
computation adds 2 steps and needs a finite program to implement\emph{
A(X) }on each parallel hypercomputer, it can be seen that a parallel
\emph{$\langle2^{2^{^{\aleph}}},\aleph,\aleph\rangle$}-hypercomputer
can compute the truth of second-order quantified propositions about
sets that require $\leq2^{\aleph}$ bits of information to define. 
\end{proof}
\begin{rem}
\label{rem:2.11} An ordinal hypercomputer is very powerful indeed;
\cite{Koepke2005,Koepke2009} show that, with a finite program and
a set of registers indexed by all bounded sets of ordinals, the class
of all ordinal computable sets of ordinals that can be computed from
finitely many ordinal parameters is Gödel's constructible set universe
\emph{L}. This result shows that with finite programs only sets of
ordinals definable by formulas in the language of set theory can be
computed using an ordinal hypercomputer. In general sets of size $\aleph$
will not be definable by a finite program, and we note that the construction
of a set of size $\aleph$ requires a serial $\langle\aleph,\aleph,\aleph\rangle$-hypercomputer
or a parallel \emph{$\langle\aleph,1,1\rangle$}-hypercomputer\footnote{This does not allow for any recursive relationships in the specification
of members of the set.}. We have seen that the class of all serial $\langle2^{2^{^{\aleph}}},2^{2^{^{\aleph}}},2^{2^{^{\aleph}}}\rangle$-hypercomputers
or parallel \emph{$\langle2^{2^{^{\aleph}}},\aleph,\aleph\rangle$}-hypercomputers
computes truth in the set theoretic universe \emph{V} for first-order
and second-order propositions of set theory with finitely many quantifiers.
If we allow a parallel hypercomputer to have $\beth\leq2^{\aleph}$
parallel hypercomputers chained together, then to compute a predicate
of length $\beth$ with $\beth$ quantifiers, the program will have
length $\beth$ and will have $\beth$ steps; hence a parallel $\langle2^{2^{^{\aleph}}},\beth,\beth\rangle$-hypercomputer
will suffice. But from the point of view of standard second-order
set theory with finitely long predicates and finitely many quantifiers,
the class of all parallel \emph{$\langle2^{2^{^{\aleph}}},\aleph,\aleph\rangle$-}hypercomputers
computes the set of all true propositions.
\end{rem}

\bibliographystyle{amsplain}
\bibliography{0_Users_andrewpowell_Downloads_hyperc1}

\end{document}